\documentclass{publicadapt}

\usepackage{amsthm}%

\allowdisplaybreaks%

{\theoremstyle{plain}%
 \newtheorem{Thm}{Theorem}

}

{\theoremstyle{remark}

}

{\theoremstyle{definition}

}

\makeatletter

\newcommand{\dotDelta}{{\vphantom{\Delta}\mathpalette\d@tD@lta\relax}}
\newcommand{\d@tD@lta}[2]{%
 \ooalign{\hidewidth$\m@th#1\mkern-1mu\cdot$\hidewidth\cr$\m@th#1\Delta$\cr}%
}
\makeatother

 \title{A further $q$-analogue of Gosper's strange series}

 \author{John M.\ Campbell}
 \address{Department of Mathematics and Statistics, Dalhousie University, 6283 Alumni Crescent, Halifax, NS, B3H 4R2, Canada}
 \email{jh241966@dal.ca}

 \author{Yuka Yamaguchi}
 \address{Faculty of Education, University of Miyazaki, 
 1-1 Gakuen Kibanadai-nishi, Miyazaki, 889-2192, Japan}
 \email{y-yamaguchi@miyazaki-u.ac.jp}
 \thanks{The second author gratefully acknowledges the support of JSPS KAKENHI Grant Number JP25KJ0266.}

 \date{Mar 24, 2026}

 \keywords{Gosper's strange series, basic hypergeometric series, $q$-analogue}
 \subjclass{33D15}

\begin{document}

 \begin{abstract}
 Recently, the second author [Ramanujan J.\ 2026] introduced and proved a $q$-series identity that appears to provide the first 
 known $q$-analogue of an evaluation for a ${}_{2}F_{1}$-series known as \emph{Gosper's strange series}. Yamaguchi's derivation of this 
 $q$-analogue relies on three-term relations for ${}_{2}\phi_{1}$-series along with Heine's transformation of ${}_{2}\phi_{1}$-series. In this 
 note, we introduce and prove, using a $q$-analogue of a 
 series evaluation technique relying on an Abel-type summation lemma, a further $q$-analogue of Gosper's ${}_{2}F_{1}$-identity that is 
 inequivalent to Yamaguchi's $q$-analogue, and we also apply this technique to construct an alternative and simplified proof of 
 Yamaguchi's $q$-analogue, together with a ${}_{3}\phi_{2}$-series variant of Heine's $q$-analogue
 of Gauss's hypergeometric formula, a ${}_{6}\phi_{5}$-series variant and two ${}_{4}\phi_{3}$-series variants 
 of the $q$-analogue of Kummer's identity due to Bailey and Daum, 
 along with 
 a $q$-analogue of a result obtained by Cantarini 
 [Ramanujan J.\ 2022] via Fourier--Legendre theory and related to Ramanujan's first series for $\frac{1}{\pi}$. 
 \end{abstract}

 \maketitle

\section{Introduction}
 Referring to Section \ref{Prelim} for preliminaries on the mathematical notation employed in this article, \emph{Gosper's strange series} 
 refers to the ${}_{2}F_{1}$-series involved in the identity 
\begin{equation}\label{labelGosper}
 {}_2F_1\!\left[
\begin{matrix}
 {1-a}, b \\
 {b+2}
\end{matrix}
; \frac{b}{a+b} 
\right] = 
 (b + 1) \left( \frac{a}{a+b} \right)^{a} 
\end{equation}
 conjectured by Gosper in 1977, with reference to a number of past research contributions on the relation in \eqref{labelGosper}
 \cite{Campbell2023,Chu2017,Ebisu2013,Yamaguchi2026}. 
 In view of the extent of the research interest 
 related to Gosper's identity \eqref{labelGosper}
 and many further series evaluations 
 and evaluation techniques due to Gosper \cite{GosperJr1976,GosperJr1978,Gosper2001,Gosper1991,Gosper1988}, 
 one may consider how the first known $q$-analogue of 
 \eqref{labelGosper} was introduced quite recently, in 2026 \cite{Yamaguchi2026}. 
 This motivates the problem of constructing 
 further $q$-analogues of \eqref{labelGosper}. 

 Yamaguchi \cite{Yamaguchi2026} introduced and proved, in an equivalent way, the remarkable 
 $q$-analogue of \eqref{labelGosper} such that 
\begin{equation}\label{firstqGosper}
{}_2\phi_1\!\left[
\begin{matrix}
q^{1-a}, q^b\\
q^{b+2}
\end{matrix}
;q,\lambda
\right] 
= \frac{1-q^{b+1}}{1-q} 
 \left[
\begin{matrix}
 \lambda q^{-a} \\
 \lambda 
\end{matrix}
;q 
\right]_{\infty}, 
\end{equation}
 where $\lambda=(1-q^b)q^{a+1} /(1-q^{a+b})$ and $| \lambda | < 1$. 
 Setting $q \to 1$ on both sides of Yamaguchi's identity, 
 we obtain an equivalent version of 
 Gosper's identity in \eqref{firstqGosper}. 
 The relation in \eqref{firstqGosper} 
 was derived using three-term relations for ${}_{2}\phi_{1}$-series, 
 building on Yamaguchi's previous work \cite{Yamaguchi2022}, 
 together with an application of Heine's transformation of ${}_{2}\phi_{1}$-series~\cite{Heine1847}. 
 Through a $q$-analogue of a series evaluation technique 
 given by Campbell and Cantarini \cite{CampbellCantarini2022} 
 and relying on a modified Abel lemma, 
 we obtain a simplified proof of 
 Yamaguchi's $q$-analogue, together with a new $q$-analogue of Gosper's identity
 that is inequivalent to 
 Yamaguchi's $q$-analogue. 
 We also show how a similar approach can be applied much more broadly in the construction
 of $q$-analogues, as in Sections \ref{secGauss}--\ref{secCantarini} below. 

\subsection{Preliminaries}\label{Prelim}
 The \emph{shifted factorial} is such that $(x)_{0} = 1$ and $(x)_{n} = x(x+1) \cdots (x + n - 1) $ for positive integers $n$. 
 \emph{Generalized hypergeometric series} may then be defined so that 
$$ 
 {}_rF_s\!\left[
\begin{matrix}
 a_{1}, a_{2}, \ldots, a_{r} \\
 b_{1}, b_{2}, \ldots, b_{s} 
 \end{matrix}
 ; x 
\right] 
 = \sum_{n=0}^{\infty} \frac{ \left( a_{1} \right)_{n} 
 \left( a_{2} \right)_{n} \cdots \left( a_{r} \right)_{n} }{ 
 \left( b_{1} \right)_{n} \left( b_{2} \right)_{n} \cdots \left( b_{s} \right)_{n} } 
 \frac{x^{n}}{n!}, 
$$ 
 referring to Bailey's classic text for background \cite{Bailey1964}. 

 For $|q|<1$, we recall the standard notation for $q$-shifted factorials. 
 For a complex number $a$, the \emph{$q$-shifted factorial} is defined by
\begin{align*}
(a;q)_0 &=1, \\
(a;q)_n &= \prod_{k=0}^{n-1}(1-aq^k), \qquad n\ge1,\\
(a;q)_\infty &= \prod_{k=0}^{\infty}(1-aq^k).
\end{align*}
 We thus obtain a $q$-analogue of the shifted factorial 
 in the sense that 
 $\lim_{q \to 1} \frac{ \left( q^{a};q \right)_{n} }{\left( 1 - q \right)^{n}}
 = \left( a \right)_{n}$. 
 \emph{Basic hypergeometric series} of the form ${}_r\phi_{r-1}$ may then be defined so that 
$$
{}_r\phi_{r-1}\!\left[
\begin{matrix}
a_1, a_2, \dots, a_r\\
b_1, b_2, \dots, b_{r-1}
\end{matrix}
;q, x
\right] 
=
\sum_{n=0}^{\infty}
 \frac{ \left( a_{1}; q \right)_{n} 
 \left( a_{2}; q \right)_{n} \cdots \left( a_{r}; q \right)_{n} }{ 
 \left( q; q \right)_{n} \left( b_{1}; q \right)_{n} \left( b_{2}; q \right)_{n} \cdots \left( b_{r-1}; q \right)_{n} } x^n, 
 $$
where it is assumed that $|q| < 1$. 
For brevity, we also write
$$ \left[
\begin{matrix}
a_1, a_2, \dots, a_r \\
b_1, b_2, \dots, b_s 
\end{matrix}
;q 
\right]_{n} = \frac{ \left( a_1; q \right)_{n} 
 \left( a_2; q \right)_{n} \cdots \left( a_r; q \right)_{n} }{ 
 \left( b_1; q \right)_{n} 
 \left( b_2; q \right)_{n} \cdots \left( b_s; q \right)_{n}}$$
 and 
$$ \left[
\begin{matrix}
a_1, a_2, \dots, a_r \\
b_1, b_2, \dots, b_s 
\end{matrix}
;q 
\right]_{\infty} = \frac{ \left( a_1; q \right)_{\infty} 
 \left( a_2; q \right)_{\infty} \cdots \left( a_r; q \right)_{\infty} }{ 
 \left( b_1; q \right)_{\infty} 
 \left( b_2; q \right)_{\infty} \cdots \left( b_s; q \right)_{\infty}}. $$
 
 For a sequence 
 $(\tau_{n} : n \in \mathbb{N}_{0})$, 
 define the backward difference operator $\nabla$ and 
 the forward difference operator $\dotDelta$ so that 
$ \nabla \tau_{n} = \tau_{n} - 
 \tau_{n-1}$ and $ \dotDelta \tau_{n} = \tau_{n} - \tau_{n + 1}$. 
 Being consistent with a number of research contributions 
 from Chu et al.\ \cite{Chu2007,ChuJia2008,ChuJia2009,ChuWang2009} that have inspired our work, we let the relation 
\begin{equation}\label{split2}
 \sum_{n = 1}^{\infty} B_{n} \nabla A_{n} 
 = \left( \lim_{m \to \infty} A_{m} B_{m + 1} \right) - A_{0} B_{1} + \sum_{n=1}^{\infty} 
 A_{n} \dotDelta B_{n} 
\end{equation}
 be referred to as the \emph{modified Abel lemma on summation by parts}, 
 with the assumption that the above limit exists
 and that the one of the two series in \eqref{split2} converges. 
 The summation lemma in \eqref{split2}
 together with a $q$-version of a 
 technique (reviewed below) based on this lemma due to Campbell and Cantarini \cite{CampbellCantarini2022} 
 provide the keys to our main results given in Section \ref{sectionMain} below. 

 For the purposes of the following discussion, we let $A_{n}$ and $B_{n}$
 be hypergeometric, but we later adapt the following method 
 to $q$-hypergeometric expressions. 
 Defining the rational function $r_{1}(n) = \frac{B_{n+1}}{B_{n}}$, 
 we then set $A_{n} = C_{n} r_2(n)$ for some sequence
 $(C_{n} : n \in \mathbb{N}_{0})$ and for a rational function 
 $r_{2}(n)$ of the form $\frac{1}{a_1 n + a_2}$ or of the form 
 $\frac{a_3 n + a_4}{a_1 n + a_2}$
 for scalars $a_i$ to be determined as follows. 
 With this setup, we rewrite the summand of the latter series in 
 \eqref{split2} so that 
\begin{equation}\label{rewritesummand}
 A_{n} \dotDelta B_{n} = B_{n} C_{n} r_2(n) (1 - r_1(n)), 
\end{equation}
 and we proceed to apply partial fraction decomposition to 
 $r_2(n) (1 - r_1(n))$, yielding a term given by a scalar multiple of
 $\frac{1}{a_1 n + a_2}$. We then set 
 $a_1$ and $a_2$ so that the highlighted term vanishes. 

The first author previously applied the above method to obtain two simplified proofs of Gosper's strange series evaluation \cite{Campbell2023}. 
In Section \ref{sectionMain}, we apply an equivalent version of a $q$-version of this method to 
construct a simplified proof of Yamaguchi's $q$-analogue in \eqref{firstqGosper} 
and to prove a new and inequivalent $q$-analogue of Gosper's identity. 
We further apply this method to obtain variants of Heine's $q$-analogue of Gauss's hypergeometric formula 
and of the $q$-analogue of Kummer's identity due to Bailey and Daum, 
as well as a $q$-analogue of a Fourier--Legendre-derived result from Cantarini~\cite{Cantarini2022}. 

 \section{Main results}\label{sectionMain}
 One may compare the much simplified proof 
 below to 
 Yamaguchi's original derivation \cite{Yamaguchi2026} of \eqref{firstqGosper}. 

\begin{Thm}
 (Yamaguchi, 2026) The $q$-analogue in \eqref{firstqGosper} of Gosper's identity holds \cite{Yamaguchi2026}. 
\end{Thm}

\begin{proof}
 In the modified Abel lemma, we set
$$
A_n = \frac{(1-q)q^n}{1-q^{b+n}}, \qquad 
B_n = \left[
\begin{matrix}
 q^{1-a} \\ 
 q 
\end{matrix}
;q 
\right]_{n-1} \left(\frac{\lambda}{q} \right)^{\,n-1}, 
$$
where $\lambda=(1-q^b)q^{a+1}/(1-q^{a+b})$ and $|\lambda|<1$. 
 Then we have 
 $$
\sum_{n=1}^{\infty} B_n \nabla A_n 
=
-\frac{(1-q)^2}{(1-q^b)(1-q^{b+1})}
\,{}_2\phi_1\!\left[
\begin{matrix}
q^{1-a}, q^b\\
q^{b+2}
\end{matrix}
;q,\lambda 
\right]. 
 $$
Since $|\lambda|<1$, we have $A_mB_{m+1}\to 0$ as $m\to\infty$.
Also, since
\begin{align*}
1-\frac{B_{n+1}}{B_n}
=
\frac{(1-q^a)(1-q^{b+n})}{(1-q^{a+b})(1-q^n)},
\end{align*}
we obtain 
 $$ 
\left(\lim_{m\to\infty}A_mB_{m+1}\right)-A_0B_1
 + \sum_{n=1}^{\infty}A_n \dotDelta B_n 
=
0-\frac{1-q}{1-q^b} 
-\frac{1-q}{1-q^b}
 \sum_{n = 1}^{\infty}
 \left[
\begin{matrix}
 q^{-a} \\ 
 q 
\end{matrix}
;q 
\right]_{n} \lambda^n .
$$
 The desired result follows by rewriting the infinite series as
\begin{align*}
\sum_{n=1}^{\infty}
 \left[
\begin{matrix}
 q^{-a} \\ 
 q 
\end{matrix}
;q 
\right]_{n} \lambda^n
 = 
 \left[
\begin{matrix}
 \lambda q^{-a} \\ 
 \lambda 
\end{matrix}
;q 
\right]_{\infty} - 1 
\end{align*}
 using the $q$-binomial theorem such that 
\begin{align*}
\sum_{n=0}^{\infty}
 \left[
\begin{matrix}
 \alpha \\ 
 q 
\end{matrix}
;q 
\right]_{n} x^n
 = 
 \left[
\begin{matrix}
 \alpha x \\ 
 x 
\end{matrix}
;q 
\right]_{\infty} 
\end{align*} 
 for $|x|<1$ and $ |q|<1$ (see, e.g., \cite[(1.3.2)]{GasperRahman2004}). 
\end{proof}

 The $q$-identity highlighted in the Theorem below
 gives us a new $q$-analogue of the Gosper's identity in \eqref{labelGosper}
 inequivalent to Yamaguchi's $q$-analogue of \eqref{labelGosper}. 
 Letting $q \to 1$ in the identity in the following Theorem, the left-hand side converges to 
\begin{equation*} 
 \frac{ a+b }{a^2 b (b+1)} \, 
 {}_2F_1\!\left[
\begin{matrix}
 {1-a}, b \\
 {b+2}
\end{matrix}
; \frac{b}{a+b} 
\right], 
\end{equation*}
 and the right-hand side converges to an expression that evaluates as 
 $$ \frac{\left(\dfrac{a}{a+b}\right)^a (a+b) }{a^2 b}, $$ 
 and we thus 
 obtain a $q$-analogue of Gosper's identity in \eqref{labelGosper}. 

\begin{Thm}\label{newGosper}
 Let $p(n) = \big(b q^{a+n}+a q^{a+n}-a q^a-b q^n \big) 
 \big( b q^{a+n+1}+a q^{a+n+1}-a q^a-b q^{n+1} \big)$. 
 Then, for $\left|\frac{b}{a+b}\right|<1$, the equality of 
 $$ (a+b) (1-q)^2 q^{2 a-1} \sum_{n=0}^{\infty} 
 \frac{\left(\frac{b q}{a+b}\right)^n}{p(n)} 
 \left[
\begin{matrix}
 q^{1 - a} \\
 q 
\end{matrix}
;q 
\right]_{n} $$
 and 
$$
 \frac{(a+b) 
 (1-q) q^{2 a-1}}{b \left(1-q^a\right) \left(a q^a-b \left(1-q^a\right)\right)} 
- \frac{1}{q \left(a+b-b q^{-a}\right)} 
 {}_2\phi_1\!\left[
\begin{matrix}
q^{1-a}, q \\
q^{2}
\end{matrix}
;q, \frac{b}{a+b}
\right] 
$$
 holds.
\end{Thm}

\begin{proof}
 We set 
 $$ A_n = \frac{1}{a_1 \frac{1-q^n}{1-q} + a_2} $$ 
 for undetermined coefficients $a_1$ and $a_2$, and we set 
 $$ B_n = 
 x^n \left[
\begin{matrix}
 \beta \\
 q 
\end{matrix}
;q 
\right]_{n}$$ for a parameter $\beta$, where $|x|<1$. 
 Using the modified Abel lemma, and using the relation in 
 \eqref{rewritesummand} for $C_{n} = 1$, 
 and applying partial fraction decomposition to 
 \eqref{rewritesummand}, we find that 
 $\sum_{n = 1}^{\infty} B_{n} \nabla A_{n}$ may be rewritten as 
$$
 -A_0 B_1 + \sum_{n=1}^{\infty} B_{n} \Bigg( 
 \frac{\beta x-q x}{\left(q^{n+1}-1\right) (a_1-a_2 q)} \\
+ \frac{v(q)}{(a_2
 q-a_1) \left(a_1 q^n-a_1+a_2 q-a_2\right)} \Bigg), 
$$
 for the value $v = v(q) = -a_1 q+a_1 \beta x-a_1 x+a_1+a_2 q^2-a_2 \beta q x-a_2 q+a_2 \beta x$. 
 By then enforcing the substitution 
 $$ a_2 = \frac{ a_1 (-q+\beta x-x+1)}{(1-q) (q-\beta x)}, $$ 
 we obtain the relation 
$$
 \sum_{n = 1}^{\infty} B_{n} \nabla A_{n} 
 \\= - A_{0} B_{1} 
 + \sum_{n=1}^{\infty} 
 B_n \frac{\beta x-q x}{\left(q^{n+1}-1\right) (a_1 - a_2 q)}. 
$$
 Setting $\beta = q^{1-a}$ and 
 $x = \frac{b}{a+b}$, we obtain an equivalent version of the desired result. 
\end{proof} 

\subsection{A variant of the $q$-analogue of Gauss' identity}\label{secGauss}
We recall the $q$-analogue of Gauss' identity due to Heine \cite{Heine1847}:
\begin{equation}\label{labelHeine}
{}_2\phi_1\!\left[
\begin{matrix}
a, b \\
c
\end{matrix}
;q, \frac{c}{ab}
\right]
=
\left[
\begin{matrix}
c/a,\, c/b \\
c,\, c/(ab)
\end{matrix}
;q
\right]_{\infty},
\qquad \left|\frac{c}{ab}\right| < 1.
\end{equation}
The following theorem presents a variant of \eqref{labelHeine}. 
By replacing $a, b, c$ with $q^a, q^b, q^c$, 
it yields a $q$-analogue of the identity in \cite[Theorem~2.1]{CampbellCantarini2022}.

\begin{Thm}
For $|c q^2/(ab)| < 1$, we have 
$$
{}_3\phi_2\!\left[
\begin{matrix}
a, b, d \\
c, d q^2
\end{matrix}
;q, \frac{cq^2}{a b}
\right] \\
= - \frac{(ab-cq) (1-d) (1-dq) q}{(1-q) (q-a) (q-b) c} 
 \left[
\begin{matrix}
 cq/a,\, cq/b \\
 c,\, cq^2/(ab) 
\end{matrix}
;q 
\right]_{\infty}, 
$$
where $d = (abq+abc-acq-bcq)/(ab-cq)q$. 
\end{Thm}

\begin{proof}
In the modified Abel lemma, we set 
$$A_n = \frac{(1-q)(ab-cq)q^{n-1}}{ab-cq-(abq+abc-acq-bcq)q^{n-1}}$$
and 
$$B_n =
\left[
\begin{matrix}
a,\, b \\
q,\, c
\end{matrix}
;q
\right]_{n-1}
\left(\frac{cq}{ab}\right)^{n-1}, $$
where $|cq^2/(ab)| < 1$. 
Then, the desired result follows from \eqref{labelHeine}, with the details omitted. 
\end{proof}

\subsection{Variants of the $q$-analogue of Kummer's identity}
We recall the $q$-analogue of Kummer's identity obtained by Bailey \cite{Bailey1941} and independently by Daum \cite{Daum1942}:
\begin{equation}\label{labelBaileyDaum}
 {}_2\phi_1\!\left[
\begin{matrix}
a, b \\
aq/b
\end{matrix}
;q, -\frac{q}{b}
\right] =
\left[
\begin{matrix}
-q \\
-q/b,\, aq/b
\end{matrix}
;q
\right]_{\infty} 
\left[
\begin{matrix}
aq,\, aq^2/b^2 \\
\text{--}
\end{matrix}
;q^2
\right]_{\infty}, \qquad \left| \frac{q}{b} \right| < 1. 
\end{equation}
The following three theorems each present a different variant of \eqref{labelBaileyDaum}. 

\begin{Thm}\label{variant1}
For $|q^2/b| < 1$, we have 
\begin{align*}
&{}_6\phi_5\!\left[
\begin{matrix}
a, a^{\frac{1}{2}}q, -a^{\frac{1}{2}}q, a^{\frac{1}{2}}q^{-\frac{1}{2}}, -a^{\frac{1}{2}}q^{-\frac{1}{2}}, b \\
a^{\frac{1}{2}}, -a^{\frac{1}{2}}, a^{\frac{1}{2}}q^{\frac{3}{2}}, -a^{\frac{1}{2}}q^{\frac{3}{2}}, aq/b
\end{matrix}
;q, -\frac{q^2}{b}
\right] \\
&\qquad 
= (1-aq) \left[
\begin{matrix}
-q^2 \\
-q^2/b,\, aq/b
\end{matrix}
;q
\right]_{\infty} 
\left[
\begin{matrix}
aq^2,\, aq^3/b^2 \\
\text{--}
\end{matrix}
;q^2
\right]_{\infty}. 
\end{align*}
\end{Thm}

\begin{proof}
In the modified Abel lemma, we set 
$$A_n = \frac{(1-q) (-q)^{n}}{1-aq^{2n-1}}$$
and 
$$B_n =
\left[
\begin{matrix}
a,\, b \\
q,\, aq/b 
\end{matrix}
;q
\right]_{n-1} \left(\frac{q}{b}\right)^{n-1}, $$
where $|q^2/b| < 1$. 
Then, the desired result follows from \eqref{labelBaileyDaum}, with the details omitted. 
\end{proof}

\begin{Thm}\label{variant2}
For $|q^2/b| < 1$, we have 
\begin{align*}
&{}_4\phi_3\!\left[
\begin{matrix}
a, b, \frac{a+b}{q+b}, \frac{(a+b)q}{1+b} \\
aq/b, \frac{(a+b)q^2}{q+b}, \frac{a+b}{1+b}
\end{matrix}
;q, -\frac{q^2}{b}
\right]  = \frac{q-aq+b-bq}{b} \left[
\begin{matrix}
-q \\
-q/b,\, aq/b
\end{matrix}
;q
\right]_{\infty} 
\left[
\begin{matrix}
aq^2,\, aq^3/b^2 \\
\text{--}
\end{matrix}
;q^2
\right]_{\infty}. 
\end{align*}
\end{Thm}

\begin{proof}
In the modified Abel lemma, we set 
$$A_n = \frac{(1-q) (-b)^{n}}{q+b-(a+b)q^n}$$
and 
$$B_n =
\left[
\begin{matrix}
a,\, b \\
q,\, aq/b 
\end{matrix}
;q
\right]_{n-1} \left(\frac{q}{b}\right)^{2(n-1)}, $$
where $|q^2/b| < 1$. 
Then, the desired result follows from \eqref{labelBaileyDaum}, with the details omitted. 
\end{proof}

\begin{Thm}\label{variant3}
For $|q/b| < 1$, we have 
\begin{align*}
&{}_4\phi_3\!\left[
\begin{matrix}
a, b, \frac{(q+b)a}{(a+b)q}, \frac{(1+b)aq}{a+b} \\
aq/b, \frac{(q+b)aq}{a+b}, \frac{(1+b)a}{a+b}
\end{matrix}
;q, -\frac{q}{b}
\right]  
= 
\frac{a+b-aq-ab}{b} \left[
\begin{matrix}
-q \\
-q/b,\, aq/b
\end{matrix}
;q
\right]_{\infty} 
\left[
\begin{matrix}
aq^2,\, aq^3/b^2 \\
\text{--}
\end{matrix}
;q^2
\right]_{\infty}. 
\end{align*}
\end{Thm}

\begin{proof}
In the modified Abel lemma, we set 
$$A_n = \frac{1-q}{a+b-(q+b)aq^{n-1}} \left(-\frac{q}{b}\right)^{n-1}$$
and 
$$B_n =
\left[
\begin{matrix}
a,\, b \\
q,\, aq/b 
\end{matrix}
;q
\right]_{n-1}, $$
where $|q/b| < 1$. 
Then, the desired result follows from \eqref{labelBaileyDaum}, with the details omitted. 
\end{proof}

The identity in Theorem~\ref{variant1} is a special case of \cite[(2.7.1)]{GasperRahman2004} with $c = a^{\frac{1}{2}} q^{-\frac{1}{2}}$ and $d = -a^{\frac{1}{2}} q^{-\frac{1}{2}}$, namely, 
$$
{}_6\phi_5\!\left[
\begin{matrix}
a, a^{\frac{1}{2}}q, -a^{\frac{1}{2}}q, b, c, d \\
a^{\frac{1}{2}}, -a^{\frac{1}{2}}, aq/b, aq/c, aq/d\
\end{matrix}
;q, \frac{aq}{bcd}
\right] 
= 
\left[
\begin{matrix}
aq,\, aq/(bc),\, aq/(bd),\, aq/(cd) \\
aq/b,\, aq/c,\, aq/d,\, aq/(bcd)
\end{matrix}
;q
\right]_{\infty}, 
$$
where $|aq/(bcd)|<1$. 
This identity becomes a $q$-analogue of Dougall's identity upon replacing $a$, $b$, $c$, and $d$ by $q^a$, $q^b$, $q^c$, and $q^d$, respectively. 
We recall Dougall's identity below (see, e.g., \cite[Corollary~3.5.2]{AndrewsAskeyRoy1999}, \cite[(2.1.7)]{GasperRahman2004}): 
\begin{align*}
&{}_5F_4\!\left[
\begin{matrix}
a, \frac{1}{2}a+1, b, c, d \\
\frac{1}{2}a, a-b+1, a-c+1, a-d+1
\end{matrix}
;1
\right] \\
&\qquad 
= \frac{\Gamma(a-b+1) \Gamma(a-c+1)}{\Gamma(a+1) \Gamma(a-b-c+1)} \frac{\Gamma(a-d+1) \Gamma(a-b-c-d+1)}{\Gamma(a-b-d+1) \Gamma(a-c-d+1)}, 
\end{align*}
where $\operatorname{Re}(a-b-c-d+1)>0$. 
Letting $d \to -\infty$, we obtain
\begin{equation}\label{labellimitDougall}
  {}_4F_3\!\left[
\begin{matrix}
a, \frac{1}{2}a+1, b, c \\
\frac{1}{2}a, a-b+1, a-c+1
\end{matrix}
;-1
\right]  
= \frac{\Gamma(a-b+1) \Gamma(a-c+1)}{\Gamma(a+1) \Gamma(a-b-c+1)}. 
\end{equation}
For this identity, see, e.g., \cite[Corollary~3.5.3]{AndrewsAskeyRoy1999}. 
By replacing $a$ and $b$ in Theorems~\ref{variant2} and \ref{variant3} with $q^a$ and $q^b$, respectively, each theorem yields a distinct $q$-analogue of \eqref{labellimitDougall} with $c = \frac{1}{2}(a - 1)$. 

\subsection{A $q$-analogue of Cantarini's formula}\label{secCantarini}
 The $\Gamma$-function may be defined via an Euler integral
 so that $\Gamma(x) = \int_{0}^{\infty} t^{x-1} e^{-t} \, dt$ 
 for $\Re(x) > 0$. In the first author's previous work on Gosper's strange series 
 \cite{Campbell2023}, it was suggested how
 the modified Abel lemma could be applied in relation to the evaluation 
$$
 \sum_{n=0}^{\infty} \left( -\frac{1}{64} \right)^{n}
 \binom{2n}{n}^{3} \frac{(4n+1)^2}{(4n-1)(4n+3)} 
 =-\frac{32(2 + \sqrt{2}) \Gamma^{2}\left( \frac{1}{4} \right) }{ \Gamma^{4}\left( \frac{1}{8} \right) } 
$$ 
 introduced by Cantarini in 2022 using relations on Clebsch--Gordan-type integrals \cite{Cantarini2022}. This is of interest in terms 
 of how the Cantarini formula above closely relates to the Bauer--Ramanujan formula $$ \sum_{n=0}^{\infty} \left( -\frac{1}{64} 
 \right)^{n} \binom{2n}{n}^{3} (4n+1) = \frac{2}{\pi} $$ and its rich history \cite{CampbellLevrie2024}. 
 This leads us to introduce a $q$-analogue of the Cantarini formula above, 
 following a similar approach as in our proof of Theorem \ref{newGosper}. 

\begin{Thm}
 Let $\rho_{1}(n) = -8 q^{2 n}-7 q^{2 n+1}-6 q^{2 n+2}-9 q^{2 n+3}-4 q^{2 n+4}-q^{2 n+5}+2 q^{2 n+6}+q^{2 n+7}+4 q^{4 n+1}+8 q^{4 n+2}+4 q^{4 n+3}+16 q$ and let 
 $\rho_2(n) $ $ =$ 
 $ -11 q^{2 n+2}$ $+$ 
 $2 q^{2 n+3}$ $-$ 
 $2 q^{2 n+5}$ $-$
 $q^{2 n+6}$ 
 $-$ 
$ q^{4 n+3} $
  $ + $ $12 q^{4 n+4}$ 
 $+$   $q^{4 n+5}$ 
 $-$
 $q^{4 n+6}$ 
 $+$
 $q^{4 n+8}$
 $-$
 $4 q^{6 n+6}$
 $+$
 $q^4$  $+$
 $q^3$  $-$
 $q^2$
 $-$
 $q+4$. Then the equality of 
$$
 \sum_{n=0}^{\infty} (-1)^{n} \left[
\begin{matrix}
 q, q, q \\ 
 q^2, q^2,q^2 
\end{matrix}
;q^2 
\right]_{n} \frac{\rho_1(n)}{\left(q^{2 n}+q^{2 n+1}-2 q\right) \left(q^{2 n+1}+q^{2 n+2}-2\right)} 
$$
 and 
\begin{align*} 
 \frac{ (1-q)(1+q)^2 q}{2} \left[
\begin{matrix}
 q, q, q \\ 
 q^2, q^2, q^2 
\end{matrix}
;q^2 
\right]_{\infty} 
&- (1+q)^2 q^2 \\
&+  \sum_{n=0}^{\infty} (-1)^n 
 \left[
\begin{matrix}
 q, q, q \\ 
 q^2, q^2, q^2 
\end{matrix}
;q^2 
\right]_{n} 
 \frac{\rho_2(n)}{\left(1-q^{n+1}\right)^3 \left(q^{n+1}+1\right)^3}
\end{align*}
 holds. 
\end{Thm}

\begin{proof}
 Setting $$ B_{n} = (-1)^{n} \left[
\begin{matrix}
 q^{\frac{1}{2}}, q^{\frac{1}{2}}, q^{\frac{1}{2}} \\ 
 q, q, q 
\end{matrix}
;q 
\right]_{n} $$ and $$A_{n} = \frac{1}{a_1 \frac{1-q^n}{1-q} + a_2}$$ in the modified Abel lemma, 
 and then setting the indeterminates $a_1$ and $a_2$ so that 
 $a_1 = 1$ and so that 
$$ a_2 = -\frac{ a_1 \left(-\sqrt{q}-2\right)}{\left(\sqrt{q}+1\right)^2 \sqrt{q}}, $$
 and then setting $q \mapsto q^2$, 
 this gives us, omitting details, an equivalent version of the desired result. 
\end{proof}

 Setting $q \to 1$ on the left-hand side, we obtain, up to a scalar multiple, a copy of the Cantarini series above. Setting $q \to 1$ on 
 the right-hand side, the resultant series, up to a scalar factor, is $$ \sum_{n=0}^{\infty} \left( -\frac{1}{64} \right)^{n} 
 \binom{2n}{n}^{3} \frac{(2 n+1) \left(4 n^2+8 n+5\right)}{(n+1)^3}, $$ and this can be evaluated in terms of the $\Gamma$-function 
 using (omitting details) the Clausen hypergeometric product identity after applying partial fraction decomposition to the rational 
 function factor in the above summand. 

\section{Conclusion}
 Omitting details, by setting $ B_{n} = x^n \frac{ \binom{2n}{n} }{16^n} $ and applying the summation technique due to Campbell
 and Cantarini \cite{CampbellCantarini2022}, an integration argument can then be used to prove evaluations such as 
  $$   \sum_{\substack{m \geq 0 \\ n \geq 1}} \frac{\binom{2n}{n}}{(m+n+1)n(n+1)} \frac{ 
 n \left( \frac{2n+1}{n+1} \right)^{m} - (n+1) \left( \frac{2n-1}{n} \right)^{m} }{8^m 16^n}  
 16 - 8 \sqrt{3} + 8\ln\left( \frac{7 \left( 2 + \sqrt{3} \right)^2}{128} \right).   $$
 We encourage a full exploration of the use of integration arguments together with the specified summation technique. Another avenue 
 to explore is given by the application of our methods 
 in relation to the $q$-series identities due to Gosper \cite{Gosper1991}
 related to the summation identity referred to as Gosper's nonlocal derangement identity.

\end{document}